\documentclass[11pt]{amsart}

\usepackage{amsmath,amssymb,amsthm,mathtools}
\usepackage{xcolor}
\usepackage[colorlinks=true,linkcolor=blue!55!black,citecolor=green!45!black,urlcolor=blue!55!black]{hyperref}

\usepackage[top=3cm, bottom=2.5cm, left=2cm, right=2cm]{geometry}

\newcommand{\seqnum}[1]{\href{https://oeis.org/#1}{\rm \underline{#1}}}
\newcommand{\TT}{\mathcal{TT}}
\newcommand{\RR}{\mathcal{R}}
\newcommand{\FF}{\mathcal{F}}
\newcommand{\JJ}{\mathcal{J}}
\newcommand{\nuu}{\nu}
\newcommand{\ind}{\mathbf{1}}

\theoremstyle{plain}
\newtheorem{theorem}{Theorem}[section]
\newtheorem{lemma}[theorem]{Lemma}
\newtheorem{proposition}[theorem]{Proposition}
\newtheorem{corollary}[theorem]{Corollary}

\theoremstyle{definition}

\newtheorem{conjecture}[theorem]{Conjecture}
\newtheorem{problem}[theorem]{Problem}
\theoremstyle{remark}

\title{Zero-Run Spectra of the $(3,2)$ Raney numbers Modulo Primes}
\author{Sen-Peng Eu}
\address{Department of Mathematics, National Taiwan Normal University, Taipei 116325, Taiwan, R.O.C.}
\author{Zai-Ting Huang}
\address{Department of Mathematics, National Taiwan Normal University, Taipei 116325, Taiwan, R.O.C.}
\author{Louis Kao}
\address{Department of Mathematics, Fu-Jen Catholic University, New Taipei City 242062, Taiwan, R.O.C.}
\date{\today}
\date{\today}
\begin{document}
\begin{abstract}

We study the zero-run structure of the $(3,2)$-Raney numbers modulo a prime $p$.
For every prime $p$, we determine the left-to-right maxima of the zero-run lengths,
the complete zero-run spectrum, and the exact number of nonzero entries in
$0\le n<p^m$.
 Interestingly, these results fall into three cases: $p=2$, $p=3$, and $p\geq5$, and the behaviors in these three cases are very different.
For $p=2$, we characterize exactly the odd terms and determine the positions
of the left-to-right maxima and the results involve Fibbinary
integers, Fibonacci numbers, and Jacobsthal numbers.
For $p=3$, we characterize exactly the nonzero terms and determine their
residues.
For $p\ge 5$, the zero runs are governed by a multiscale system of residue
intervals modulo powers of $p$, from which both the record values and the
complete zero-run spectrum are obtained.

\end{abstract}
\maketitle
\section{Introduction}\label{sec:intro}
\subsection{The Raney numbers $R_{3,2}$}
The Raney numbers

$$
R_{k,r}(n)=\frac{r}{kn+r}\binom{kn+r}{n}
$$
play an important role in combinatorics~\cite{R_60}. For example, the ubiquitous Catalan numbers are just $R_{2,1}(n)$. 

Congruence properties of the Catalan numbers are extensively studied in the literature; see~\cite{AK_73,DS_06,S_03} for examples. In particular, Alter and Kubota~\cite{AK_73} studied the lengths and positions of blocks of Catalan numbers divisible by a prime, as well as their divisibility by prime powers. More precisely, 
for a prime $p$, let $B_k^{(p)}$ denote the $k$-th maximal block of consecutive Catalan numbers divisible by $p$, with length $L_k^{(p)}$. For $p=2$, they showed that $L_k^{(2)}=2^k-1$. For an odd prime $p$, let $m$ be the largest nonnegative integer such that
$\left(\frac{p+1}{2}\right)^m\mid k$, then
$$L_k^{(p)}=
\begin{cases} 
\frac{3^{m+2}-3}{2}, &\quad p=3 \\
\frac{p^{m+1}-3}{2}, &\quad p\ge 5
\end{cases} 
$$
 They also determined explicitly the positions of the blocks $B_k^{(p)}$.

It is then natural to seek analogous results for other Raney numbers. Although congruence properties of Raney numbers have been studied to some extent. For example, Bobrowski, He, and Shiue~\cite{BHS_19} obtained divisibility results for certain classes of Raney numbers, while more recently Krattenthaler and M\"{u}ller~\cite{KM_25} developed a general method for studying combinatorial sequences modulo prime powers and, among other applications, determined the modulo $p^k$ behavior of certain Fuss-Catalan numbers. These works are primarily concerned with divisibility and congruence properties. By contrast, the structure of maximal zero runs for Raney numbers appears to remain largely unexplored.

This is the focus of the present work. In this paper, we focus on the sequence \seqnum{A006013}

$$
\TT_n:=R_{3,2}(n)=\frac{1}{n+1}\binom{3n+1}{n}
=1,2,7,30,143,728,3876,21318,120175,\ldots.
$$

The importance of $\TT_n$ in enumerative combinatorics is that it occurs naturally in several core combinatorial structures. It directly counts ordered pairs of ternary trees (hence our notation) with a prescribed total number of internal nodes~\cite{B_16,EFP_26}. It also counts the total number of tails among all fighting fish of size $n+1$, and the number of horizontally symmetric fighting fish of size $2n+1$~\cite{EFP_26}. Moreover, fighting fish are in bijection with left ternary trees, two-stack-sortable permutations, non-separable rooted planar maps, and synchronized Tamari intervals, each of which is a core combinatorial structure. See~\cite{DGRS_16,DH_23,EFP_26,F_17,W_93} and the references therein. In fact, these relations and their refinements are still an active topic of research.

\subsection{Our goal}

Our goal is to investigate the congruence properties of $\TT_n \bmod p$, with particular emphasis on maximal consecutive strings of zero terms.

Fix a prime $p$. A \emph{zero run} of $\TT_n\bmod p$ is a maximal interval of consecutive indices on which $\TT_n\equiv 0\pmod p$. We number the zero runs from left to right and denote by $\RR_i^{(p)}$ the length of the $i$-th zero run, with $i\ge 1$. 

For a sequence $\langle a_i\rangle_{i\geq 1}$, we say that $a_k$ is a \emph{left-to-right maximum}, or a \emph{record}, if
$$
a_k\geq a_i\qquad\text{for every }i<k.
$$
Thus ties are allowed.
For $m\geq 1$, define the \emph{zero-run enumerator} of $\TT_n\bmod p$ on $0\leq n<p^m$ by
$$
Z_{p,m}(y)=\sum_R y^{|R|},
$$
where the sum runs over all maximal zero runs $R$ in the finite word
$\TT_0,\TT_1,\ldots,\TT_{p^m-1}\pmod p$
and $|R|$ is the length of $R$.
Note that a zero run meeting the right endpoint $p^m-1$ is counted with its truncated length inside this interval. 

Our goals are the following:
\begin{itemize}
\item For every prime $p$, determine the values of the records of the sequence $\langle \RR_i^{(p)}\rangle_{i\geq 1}$;

\item For every prime $p$, determine the zero-run enumerator $Z_{p,m}(y).$
\end{itemize}

In this paper, we completely solve the above two problems for all primes. Interestingly, the results fall into three cases: $p=2$, $p=3$, and $p\geq 5$, and the structures in these three cases are very different.
For $p=2$, the support (nonzero terms) is characterized by even Fibbinary integers, and the record structure is described by Fibonacci and Jacobsthal numbers.
For $p=3$, the support is characterized by weakly increasing ternary words. For $p\ge5$, the divisibility condition is governed instead by a hierarchy
of residue intervals modulo $p,p^2,p^3,\ldots$.

We remark that for $p=2,3$ some results on the zero run sequence appear in OEIS without proofs, see \seqnum{A085407} for example. For completeness, we also give proofs.

\subsection{An example for $p=7$}
Take $p=7$. The first $343$ terms ($0\le n\le 7^3-1$) of $\TT_n\bmod 7$ are
1, 2, 0, 2, 3, 0, 5, 3, 6, 0, 1, 5, 0, 4, 1, 2, 0, 0, 0, 0, 0, 0, 0, 0, 2, 3, 0, 6, 5, 3, 0, 2, 3, 0, 0, 0, 0, 0, 0, 0, 0, 0, 0, 0, 0, 0, 0, 0, 5, 3, 6, 0, 6, 2, 0, 1, 2, 4, 0, 3, 1, 0, 5, 3, 6, 0, 0, 0, 0, 0, 0, 0, 0, 1, 5, 0, 3, 6, 5, 0, 1, 5, 0, 0, 0, 0, 0, 0, 0, 0, 0, 0, 0, 0, 0, 0, 0, 4, 1, 2, 0, 2, 3, 0, 5, 3, 6, 0, 1, 5, 0, 4, 1, 2, 0, 0, 0, 0, 0, 0, 0, 0, 0, 0, 0, 0, 0, 0, 0, 0, 0, 0, 0, 0, 0, 0, 0, 0, 0, 0, 0, 0, 0, 0, 0, 0, 0, 0, 0, 0, 0, 0, 0, 0, 0, 0, 0, 0, 0, 0, 0, 0, 0, 0, 0, 0, 0, 0, 0, 0, 0, 2, 3, 0, 6, 5, 3, 0, 2, 3, 0, 0, 0, 0, 0, 0, 0, 0, 0, 0, 0, 0, 0, 0, 0, 6, 5, 3, 0, 3, 1, 0, 4, 1, 2, 0, 5, 4, 0, 6, 5, 3, 0, 0, 0, 0, 0, 0, 0, 0, 2, 3, 0, 6, 5, 3, 0, 2, 3, 0, 0, 0, 0, 0, 0, 0, 0, 0, 0, 0, 0, 0, 0, 0, 0, 0, 0, 0, 0, 0, 0, 0, 0, 0, 0, 0, 0, 0, 0, 0, 0, 0, 0, 0, 0, 0, 0, 0, 0, 0, 0, 0, 0, 0, 0, 0, 0, 0, 0, 0, 0, 0, 0, 0, 0, 0, 0, 0, 0, 0, 0, 0, 0, 0, 0, 0, 0, 0, 0, 0, 0, 0, 0, 0, 0, 0, 0, 0, 0, 0, 0, 0, 0, 0, 0, 0, 0, 0, 0, 0, 0, 0, 0, 0, 0, 0, 0, 0, 0, 0, 0, 0, 0, 0, 0, 0, 0, 0, 0, 0, 0, 0, 5$\ldots$

The sequence of zero-run lengths begins
$$\RR_i^{(7)}=1,1,1,1,8,1,1,15,1,1,1,1,8,1,1,15,1,1,1,1,57,\ldots $$

We will prove that the sequence of records $1,8,15,57,113,400,799,\dots$ are obtained by interlacing  the sequence $\langle \frac{7^r-4}{3}\rangle_{r\ge 1}=1, 15, 113, 799, 5601$ and 
$\langle \frac{7^r-1}{6}\rangle_{r\ge 2}=8, 57, 400, 2801,\dots$.

 Also, we have
$$Z_{7,3}(y)=24y+3y^8+3y^{15}+y^{57}+y^{113},$$
which means that, for $0\leq n\le 7^3-1$, the sequence $\TT_n\bmod 7$ contains 24 zero runs of length $1$, three of length $8$, three of length $15$, one of length $57$, and one of length $113$.

\bigskip

The rest of the paper is organized as follows. In Section~\ref{sec:prelim}, we introduce the notation and collect the basic valuation and carry lemmas used throughout the paper. In Section~\ref{sec:p2}, we treat the case $p=2$: we characterize the odd terms, determine the Fibonacci-Jacobsthal record structure, and obtain the complete zero-run enumerator. In Section~\ref{sec:p3}, we treat the case $p=3$ and determine both the nonzero residues and the complete zero-run spectrum. In Section~\ref{sec:pge5}, we treat all primes $p\geq 5$ by a multiscale residue-interval analysis and obtain the record values and spectrum. In Section~\ref{sec:support}, we count the nonzero entries in $0\leq n<p^m$ for every prime $p$. We conclude in Section~\ref{sec:conclusion} with several remarks and open directions.


\section{Definitions and basic lemmas}\label{sec:prelim}

We collect the notation and elementary tools.  Denote by $\FF_0=0,\FF_1=1$ and $\FF_{n+1}=\FF_n+\FF_{n-1}$ the \emph{Fibonacci numbers}, and denote by $\JJ_0=0,\JJ_1=1$ and $\JJ_{n+1}=\JJ_n+2\JJ_{n-1}$ the \emph{Jacobsthal numbers}. A \emph{Fibbinary number} is a nonnegative integer whose binary expansion contains no consecutive $1$'s.  The \emph{support} of a sequence is the set of indices with nonzero terms. 
For a statement $P$, we set $\ind_{P}=1$ if $P$ is true, and $0$ if false.

For a prime $p$ and a positive integer $N$, let $\nuu_p(N)$ denote the exponent of the largest power of $p$ dividing $N$. We use the convention $\nuu_p(1)=0$.
It is clear that 
\begin{equation}~\label{vp}
\nuu_p(\TT_n)=\nuu_p((3n+1)!)-\nuu_p((2n+1)!)-\nuu_p((n+1)!)
\end{equation}
for every prime $p$.

The main tool is Legendre's formula~\cite{L_30}.
\begin{lemma}[Legendre's formula]\label{lem:Legendre}
Let $n=(\mathtt{a_k a_{k-1}\cdots a_0})_p$ be the base-$p$ expansion of $n$. Then
$$\nuu_p(n!)=\sum_{j\geq 1}\left\lfloor\frac{n}{p^j}\right\rfloor=\frac{n-\sum_{i=0}^k\mathtt{a_i}}{p-1}.$$
\end{lemma}

This gives a formulation that is convenient for the carry arguments.

\begin{lemma}\label{lem:carry}
Start from the formal base-$p$ expression $(\mathtt{0\cdots 0}\,\mathtt{n})_p$, in which the last ``digit'' is allowed to be larger than $p-1$, and normalize it to the proper base-$p$ expansion of $n$ by elementary carries. If a carry of size one is counted each time $p$ is removed from one digit and $1$ is added to the next digit, then the total number of carries is $\nuu_p(n!)$.
\end{lemma}

\begin{proof}
The initial digit sum is $n$, whereas the digit sum of the proper base-$p$ expansion is $s_p(n):=\sum_i\mathtt{a_i}$. Each elementary carry decreases the digit sum by exactly $p-1$. Hence the number of carries is
$$\frac{n-s_p(n)}{p-1}=\nuu_p(n!)$$
by Lemma~\ref{lem:Legendre}.
\end{proof}

We also need a digitwise version as follows. For nonnegative integers $a,b$, if $n=(\mathtt{a_k\cdots a_0})_p$, let $c_p(an+b)$ be the number of elementary carries needed to normalize
$$((a\mathtt{a_k})(a\mathtt{a_{k-1}})\cdots(a\mathtt{a_1})(a\mathtt{a_0}+b))_p$$
to the proper base-$p$ expansion of $an+b$.

\begin{lemma}\label{lem:anb}
If $\nuu_p(n!)=\ell$, then
$$\nuu_p((an+b)!)=a\ell+c_p(an+b).$$
\end{lemma}

\begin{proof}
Before normalization, the digit sum is $a s_p(n)+b$. Since every elementary carry lowers the digit sum by $p-1$, we have
$$s_p(an+b)=a s_p(n)+b-(p-1)c_p(an+b).$$
Applying Lemma~\ref{lem:Legendre} gives
\begin{align*}
\nuu_p((an+b)!)
&=\frac{an+b-s_p(an+b)}{p-1}\\
&=a\frac{n-s_p(n)}{p-1}+c_p(an+b)\\
&=a\nuu_p(n!)+c_p(an+b).
\end{align*}
\end{proof}

Combining (\ref{vp}) with Lemma~\ref{lem:anb} we have the following proposition.

\begin{proposition}\label{prop:carryTT}
For every prime $p$,
\[
\nuu_p(\TT_n)=c_p(3n+1)-c_p(2n+1)-c_p(n+1).
\]
\end{proposition}

\begin{proof}
If $\ell=\nuu_p(n!)$, then Lemma~\ref{lem:anb} gives
\begin{align*}
\nuu_p(\TT_n)
&=(3\ell+c_p(3n+1))-(2\ell+c_p(2n+1))-(\ell+c_p(n+1))\\
&=c_p(3n+1)-c_p(2n+1)-c_p(n+1).
\end{align*}
\end{proof}

For what follows, we refer to the modulus $p^m$ as \emph{the $m$-th $p$-adic scale} (or simply the \emph{scale $p^m$}). For example, a condition at scale $125$ means that we are considering modulo $125$.

\section{The case \texorpdfstring{$p=2$}{p=2}}\label{sec:p2}

We first determine exactly when $\TT_n$ is odd. 
We will show that the nonzero indices in $\TT_n\bmod2=1,0,1,0,1,0,0,0,1,0,1,0,0,0,0,0,1,0,1,0,1,0,0, \dots$ are the even Fibbinary integers
$0,2,4,8,10,16,18,20, \ldots.$
\begin{theorem}\label{thm:mod2}
The number $\TT_n$ is odd if and only if $n$ is even and its binary expansion contains no consecutive $1$'s. Equivalently, $n$ is an even Fibbinary integer.
\end{theorem}

\begin{proof}
Write $n=(\mathtt{a_k a_{k-1}\cdots a_1a_0})_2.$
By Proposition~\ref{prop:carryTT}, $\TT_n$ is odd if and only if
\[
c_2(3n+1)=c_2(2n+1)+c_2(n+1).
\]
Before the final normalization, the three relevant binary expressions can be written as
\begin{align*}
n+1&=(\mathtt{a_k a_{k-1}\cdots a_1}(\mathtt{a_0}+1))_2,\\
2n+1&=(\mathtt{a_k a_{k-1}\cdots a_1a_0 1})_2,\\
3n+1&=(\mathtt{a_k}(\mathtt{a_k}+\mathtt{a_{k-1}})\cdots
(\mathtt{a_1}+\mathtt{a_0})(\mathtt{a_0}+1))_2.
\end{align*}

Let $s_2(n)$ be the binary digit sum of $n$, and let $d_2(3n+1)$ denote the number of carries needed to normalize the last displayed expression for $3n+1$. From the formal expressions in the definition of $c_2$, we have
$c_2(2n+1)=s_2(n)$ and $c_2(3n+1)=s_2(n)+d_2(3n+1)$.
Hence from Proposition 2.4 we have
$$\nu_2(\TT_n)=d_2(3n+1)-c_2(n+1).$$
Thus it suffices to compare the carries coming from the sums $a_i+a_{i-1}$ and the terminal digit $a_0+1$ with those occurring in $n+1$.

If $\mathtt{a_0}=0$ and no two consecutive digits of $n$ are both $1$, then all adjacent sums are at most $1$ and the last digit is $1$. Hence no additional carry occurs, and $\nuu_2(\TT_n)=0$.

Conversely, suppose that $n$ is not an even Fibbinary integer.  If $n$ is even, then $\mathtt{a_0}=0$ and some adjacent pair is $\mathtt{11}$. The corresponding digit in the expression for $3n+1$ is at least $2$, producing at least one additional carry that does not occur in $n+1$ or $2n+1$. Thus $\nuu_2(\TT_n)\geq1$, a contradiction.

Otherwise, if $n$ is odd, let $r\geq1$ be the number of trailing $1$'s in its binary expansion. Then adding $1$ to $n$ produces exactly $r$ carries. In the expression for $3n+1$, the terminal digit $\mathtt{a_0}+1=2$ starts a carry chain through these trailing $1$'s and forces at least one further carry at the first digit to their left. Consequently
\[
c_2(3n+1)>c_2(2n+1)+c_2(n+1),
\]
so $\nuu_2(\TT_n)\geq1$, again a contradiction.
Hence the result is proved.
\end{proof}

We next determine the record (left-to-right maximum) of the zero run sequence $\mathcal{R}_i^{(2)}$. 

\begin{theorem}\label{thm:Jacobsthal}
For $i\geq1$, the entry $\RR_i^{(2)}$ is a left-to-right maximum if and only if
$i=\FF_k$
for some $k\geq2$. Moreover,
\[
\RR_{\FF_k}^{(2)}=\JJ_{k-1}.
\]
\end{theorem}

\begin{proof}
For $m\geq1$, let $S_m$ be the set of integers $0\leq n<2^m$ for which $\TT_n$ is odd. By Theorem~\ref{thm:mod2}, $S_m$ is the set of binary words of length $m$ that end in $0$ and contain no consecutive $1$'s. Hence
$|S_m|=\FF_{m+1}$.

These \emph{admissible words} satisfy the recursive decomposition
$$S_m=S_{m-1}\ \cup\ \bigl(2^{m-1}+S_{m-2}\bigr)$$
when $m\ge 3$.

Let us consider $\max S_m$. The largest admissible word is obtained greedily by alternating $1$'s and $0$'s from the most significant digit (subject to the final digit being $0$). A direct calculation gives
$$2^m-1-\max S_m=\frac{2^m-(-1)^m}{3}=\JJ_m.$$
Thus the terminal zero run in $0\leq n<2^m$ has length $\JJ_m$.

Inductively the above decomposition shows that every earlier zero run has length at most $\JJ_{m-1}$, while the gap between the two displayed blocks has length exactly $\JJ_{m-1}$. Since $\JJ_m>\JJ_{m-1}$ for $m\geq3$, the terminal run at scale $2^m$ is the new record run. Its position in the zero-run sequence equals the number of nonzero entries in $0\leq n<2^m$, which is $\FF_{m+1}$. Therefore
$$\RR_{\FF_{m+1}}^{(2)}=\JJ_m$$
and the result follows from the initials $\RR_1^{(2)}=\RR_2^{(2)}=1$.
\end{proof}

The recursive decomposition leads to the complete spectrum $Z_{2,m}(y)$.

\begin{theorem}\label{thm:spec2}
We have $Z_{2,1}(y)=y$, and for $m\geq2$,
\[
Z_{2,m}(y)=\FF_m y+\sum_{j=3}^{m-1}\FF_{m-j}y^{\JJ_j}+y^{\JJ_m}.
\]
\end{theorem}

\begin{proof}
For $m\geq3$, use the decomposition
$S_m=S_{m-1}\cup\bigl(2^{m-1}+S_{m-2}\bigr).$
The zero runs determined by the first block contribute $Z_{2,m-1}(y)$. In the second block, all nonterminal runs are copies of the runs at scale $2^{m-2}$. The terminal run of length $\JJ_{m-2}$ at that smaller scale is not preserved as a terminal run; it merges with the final part of the interval and becomes the new terminal run of length $\JJ_m$. Therefore
$$Z_{2,m}(y)=Z_{2,m-1}(y)+Z_{2,m-2}(y)-y^{\JJ_{m-2}}+y^{\JJ_m}$$
and the recurrence yields the result by induction on $m$,
with the initials cases $Z_{2,1}(y)=y$, $Z_{2,2}(y)=2y$, $Z_{2,3}(y)=2y+y^3$.
\end{proof}

\begin{corollary}\label{cor:p2gf}
The bivariate generating function in the scale variable is
\[
\sum_{m\geq1}Z_{2,m}(y)t^m
=\frac{1-t^2}{1-t-t^2}\sum_{j\geq1}y^{\JJ_j}t^j.
\]
\end{corollary}

\begin{proof}
This can be done by simple calculation: just multiply the recurrence in the proof of Theorem~\ref{thm:spec2} by $t^m$, sum over $m\geq3$, and use the three initial values above.
\end{proof}
We will see later that for $p\ge 3$, there is no such clean form as above.


\section{The case \texorpdfstring{$p=3$}{p=3}}\label{sec:p3}

We will see the $p=3$ case has a different structure and the support can also be characterized explicitly in terms of the base-$3$ digits. For this, we look closely at the carry mechanism.

\begin{lemma}[Ternary carry lemma] \label{lem:ternarycarry}
Let $n=(\mathtt{a_k a_{k-1}\cdots a_1a_0})_3$ with
$\mathtt{a_i}\in\{0,1,2\}$, and let $r$ be the number of trailing $2$'s in the base-$3$ expansion of $n$ (thus
$\mathtt{a_0}=\cdots=\mathtt{a_{r-1}}=2$, and either $r=k+1$, or $\mathtt{a_r}\in\{0,1\}$).
Define $\epsilon_0,\epsilon_1,\ldots,\epsilon_k\in\{0,1\}$ recursively by
$\epsilon_0=0 (\text{if } \mathtt{a_0}=0) \text{or } \epsilon_0=1 (\text{if } \mathtt{a_0}=1\text{ or }2$), and for $i\geq1$,
$$\epsilon_i=
\begin{cases}
0,&\mathtt{a_i}=0,\\
\epsilon_{i-1},&\mathtt{a_i}=1,\\
1,&\mathtt{a_i}=2.
\end{cases}
$$
Then $$\nuu_3(\TT_n)=\sum_{i=0}^k\mathtt{a_i}-\sum_{i=0}^k\epsilon_i-r=\sum_{i=r}^k(\mathtt{a_i}-\epsilon_i),$$
where the last sum is empty when $r=k+1$. 
\end{lemma}

\begin{proof}
We compute the three terms in Proposition~\ref{prop:carryTT} separately:
\begin{enumerate}
\item First we consider $3n+1$. Before normalization we have
$3n+1=((3\mathtt{a_k})(3\mathtt{a_{k-1}})\cdots (3\mathtt{a_1})(3\mathtt{a_0}+1))_3.$
At the units digit, the outgoing carry is $\mathtt{a_0}$. Inductively, if the carry entering the $i$-th digit is $\mathtt{a_{i-1}}$, then
$3\mathtt{a_i}+\mathtt{a_{i-1}}$
produces the outgoing carry $\mathtt{a_i}$ since $0\leq\mathtt{a_{i-1}}\leq2$. Hence
$$c_3(3n+1)=\sum_{i=0}^k\mathtt{a_i}.$$

\item Next we consider $2n+1$. Let $\epsilon_i$ be the carry leaving the $i$-th digit during the normalization of
$((2\mathtt{a_k})(2\mathtt{a_{k-1}})\cdots (2\mathtt{a_1})(2\mathtt{a_0}+1))_3.$
Since every incoming carry is either $0$ or $1$, we have
$\epsilon_0=\left\lfloor\frac{2\mathtt{a_0}+1}{3}\right\rfloor$
and, for $i\geq1$, 
$\epsilon_i=\left\lfloor \frac{2\mathtt{a_i}+\epsilon_{i-1}}{3}\right\rfloor$.
These formulas are exactly the recursion in the statement and therefore
$$c_3(2n+1)=\sum_{i=0}^k\epsilon_i.$$
\item Finally we consider $n+1$. In the normalization of $n+1$, a carry leaves the $i$-th digit precisely when
$\mathtt{a_0}=\mathtt{a_1}=\cdots=\mathtt{a_i}=2$.
Consequently,
$$c_3(n+1)=r.$$
\end{enumerate}

Applying proposition~\ref{prop:carryTT} we have
$\nuu_3(\TT_n)=\sum_{i=0}^k\mathtt{a_i}-\sum_{i=0}^k\epsilon_i-r$. However for $0\leq i<r$, we have
$\mathtt{a_i}=2$, $\epsilon_i=1$,
so the contribution of these $r$ digits to
$\sum_i(\mathtt{a_i}-\epsilon_i)-r$
is zero. Hence
$$\nuu_3(\TT_n)=\sum_{i=r}^k(\mathtt{a_i}-\epsilon_i).$$
Note that each summand on the right is nonnegative (if $\mathtt{a_i}=0$, then $\epsilon_i=0$; if $\mathtt{a_i}=1$, then $\epsilon_i\in\{0,1\}$; and if $\mathtt{a_i}=2$, then $\epsilon_i=1$) and the lemma is proved.

\end{proof}

\begin{theorem}\label{thm:mod3}
The number $\TT_n$ is not divisible by $3$ if and only if the base-$3$ representation of $n$ is weakly increasing from left to right. More precisely,
\[
\TT_n\equiv
\begin{cases}
1\pmod3,& n=0\text{ or }n=(\mathtt{2\cdots2})_3,\\
2\pmod3,& n=(\underbrace{\mathtt{1\cdots1}}_{i}\underbrace{\mathtt{2\cdots2}}_{j})_3 \text{ with }i\geq1,\ j\geq0,\\
0\pmod3,&\text{otherwise.}
\end{cases}
\]
\end{theorem}

\begin{proof}
To see when $\nuu_3(\TT_n)=0$ it suffices to determine when all summands in the last lemma are zeros. If $r>0$, the trailing block of $2$'s supplies an incoming carry to the next digit; if $r=0$ and $\mathtt{a_0}=1$, the added $1$ supplies the same initial carry. Thereafter a digit $1$ preserves an incoming carry, a digit $0$ kills it, and a digit $2$ creates a carry. Thus equality $\mathtt{a_i}-\epsilon_i=0 (i\geq r)$ forces the digits, read from right to left, to have the form
$\mathtt{2}^r\mathtt{1}^s\mathtt{0}^t$ for some $r,s,t\ge 0$ and the first part is proved.

It remains to determine the residue of $\TT_n$ modulo $3$. The key observation is the following. 
For a factorial $N!$, we remove all powers of $3$ from its factors and reduce the remaining unit parts modulo $3$. Now each unit part is $1$ or $2$, and the residue is determined by the parity of the number of factors whose $3$-free part is congruent to $2$, but in base $3$ these are exactly the positive integers whose last nonzero ternary digit is $2$. Now we look at the numbers whose base-$3$ expansion is weakly increasing case by case. 
\begin{enumerate}
\item First, let $n=(\mathtt{1\cdots1})_3$ with $i\geq1$ digits. Counting such unit parts in $(3n+1)!$, $(2n+1)!$, and $(n+1)!$ gives respectively
$\frac{3^{i+1}-3-2i}{4}$, $\frac{3^i-1}{2}$, and $\frac{3^i+3-2i}{4}$. One can see that their signed difference is $-1$, so $\TT_n\equiv2^{-1}\equiv2\pmod3$.
\item For $n=(\underbrace{\mathtt{1\cdots1}}_i\underbrace{\mathtt{2\cdots2}}_j)_3$ with $i,j\geq1$,
the corresponding three counts are $\frac{3^{i+j+1}+3^{j+1}-6-2i}{4}$, $\frac{3^{i+j}+3^j-2}{2}$, and $\frac{3^{i+j}+3^j+2-2i}{4}$.
Again their signed difference is $-1$, and hence $\TT_n\equiv2\pmod3$. 
\item For $n=(\mathtt{2\cdots2})_3$ with $i\geq1$ digits, the three counts are
$\frac{3^{i+1}-3}{2}$, $3^i-1$, and $\frac{3^i-1}{2}$, whose signed difference is $0$. Thus $\TT_n\equiv1\pmod3$. 
\end{enumerate}
Combining these cases completes the proof.
\end{proof}

For example, among $0\leq n<27$ the nonzero indices of $\TT_n \bmod 3$ are the ternary words of length three that are weakly increasing:
$\mathtt{000},\mathtt{001},\mathtt{002},\mathtt{011},\mathtt{012},\mathtt{022},\mathtt{111},\mathtt{112},\mathtt{122},\mathtt{222}$ and the corresponding indices are $0,1,2,4,5,8,13,14,17,26$.

We now read the zero runs directly from the ordered list of admissible ternary words.

\begin{theorem}\label{thm:Run3}
The set of values attained by left-to-right maxima of $\RR_i^{(3)}$ is
\[
\left\{\frac{3^r-1}{2}:r\geq1\right\}
\cup
\left\{3^r-1:r\geq1\right\}.
\]
\end{theorem}

\begin{proof}
By Theorem~\ref{thm:mod3}, the positive nonzero indices with exactly $i$ ternary digits are
$$u_{i,j}:=(\underbrace{\mathtt{1\cdots1}}_{i-j} \underbrace{\mathtt{2\cdots2}}_j)_3=\frac{3^i+3^j-2}{2}$$
for $0\leq j\leq i$.

The number of zeros between two consecutive nonzero indices $u_{i,j-1}$ and $u_{i,j}$.
For $j=1$ the gap has length $0$ and there is no zero run. 
For $2\le j\le i$ the corresponding zero-run length is $u_{i,j}-u_{i,j-1}-1=3^{j-1}-1$.

The gap between the last nonzero index of $i$ digits, namely $u_{i,i}=3^i-1$, and the first nonzero index of $i+1$ digits, namely $u_{i+1,0}=(3^{i+1}-1)/2$, has length
$u_{i+1,0}-u_{i,i}-1=\frac{3^i-1}{2}$.
These formulas list every zero run. As the scale increases, the new largest internal gap is $3^i-1$, and the cross-scale gap is $(3^i-1)/2$. Therefore, together they give exactly the stated record values.
\end{proof}

For example, by interlacing $\frac{3^r-1}{2}=1,4,13,40,\ldots$ and $3^r-1=2,8,26,80,\ldots$
we obtain the first distinct record values $1,2,4,8,13,26,40,80,\ldots$ of $\RR_i^{(3)}$.

We can also determine the record positions.
\begin{corollary}
For every $r\ge 1$, we have
$$R^{(3)}_{\frac{r(r+1)}{2}}=\frac{3^r-1}{2}, \qquad R^{(3)}_{\frac{r(r+3)}{2}}=3^r-1.$$
\end{corollary}

\begin{proof}
By the proof of Theorem~4.3, among the nonzero indices with exactly
$i$ ternary digits, the positive internal gaps have lengths
$3^{j-1}-1$, $2\le j\le i,$ and the gap from the last such index to the first nonzero index with
$i+1$ ternary digits has length
$\frac{3^i-1}{2}$.
Thus the $i$-th block contributes $i$ zero run, namely, $i-1$ internal ones
followed by one cross-scale run.
Hence the cross-scale run following the $r$-th block is the
$1+2+\cdots+r=\frac{r(r+1)}{2}$-th zero run, and therefore
$R^{(3)}_{\frac{r(r+1)}{2}}=\frac{3^r-1}{2}$.

Also, the first internal gap of length $3^r-1$ occurs in the
$(r+1)$-st block, as its last internal gap. There are
$r(r+1)/2$ zero runs before this block and $r$ positive internal gaps
up to and including this one. Hence its position is
$\frac{r(r+1)}{2}+r=\frac{r(r+3)}{2}$, which gives $R^{(3)}_{\frac{r(r+3)}{2}}=3^r-1$.
\end{proof}

The full enumerator of zero-run lengths can be obtained readily.

\begin{theorem}\label{thm:spec3}
For every $m\geq1$, we have 
$$Z_{3,m}(y)=\sum_{i=1}^{m-1}y^{(3^i-1)/2}+\sum_{j=2}^{m}(m-j+1)y^{3^{j-1}-1}.$$
\end{theorem}

\begin{proof}
For each fixed $i$, the internal gaps between $u_{i,j-1}$ and $u_{i,j}$ have length $3^{j-1}-1$. For a fixed $j\geq2$, this gap occurs once for every $i=j,j+1,\ldots,m$, hence with multiplicity $m-j+1$. This gives
\[
\sum_{j=2}^{m}(m-j+1)y^{3^{j-1}-1}.
\]
The gap from $u_{i,i}$ to $u_{i+1,0}$ has length $(3^i-1)/2$ and occurs once for each $1\leq i\leq m-1$, giving
\[
\sum_{i=1}^{m-1}y^{(3^i-1)/2}.
\]
There are no other gaps between consecutive nonzero terms, so the two contributions give the result.
\end{proof}

For example, the nonzero indices below $27$ are $0,1,2,4,5,8,13,14,17,26$ and hence the zero runs have lengths
$1,2,4,2,8$, whose generating function is $Z_{3,3}(y)=y+2y^2+y^4+y^8$.


\section{The case \texorpdfstring{$p\geq5$}{p at least 5}}\label{sec:pge5}
We will show that, for primes $p\geq5$, rather than a comparably simple description of $n$ with $\nu_p(\TT_n)=0$, Legendre's formula instead turns the valuation into a multiscale family of intervals. For an integer $q\geq3$, let $[x]_q$ denote the least nonnegative residue of $x$ modulo $q$. Define the following union of interval of integers
$$I_q:=\left[\left\lceil\frac{q+2}{3}\right\rceil,\frac{q-1}{2}\right] \cup \left[\left\lceil\frac{2q+2}{3}\right\rceil,q-1\right].$$
For example, we have $I_5=[4]$, $I_{25}=[9,12]\cup[18,24]$, and $I_{125}=[43,62]\cup [84,124]$. Also
$I_7=[3]\cup [6]$, $I_{49}=[17, 24]\cup [34, 48]$ and $I_{343}=[115, 171]\cup [230, 342]$. 
The following is the key technical lemma for the proof.

\begin{lemma}\label{lem:Interval}
Let $q\geq3$ be odd and let $x$ be a positive integer. If $r=[x]_q$, then
$$\left\lfloor\frac{3x-2}{q}\right\rfloor -\left\lfloor\frac{2x-1}{q}\right\rfloor -\left\lfloor\frac{x}{q}\right\rfloor =\ind_{r\in I_q}.$$
\end{lemma}

\begin{proof}
By direct computation. Write $x=aq+r$ with $0\leq r<q$. The left-hand side becomes
$$\left\lfloor\frac{3r-2}{q}\right\rfloor -\left\lfloor\frac{2r-1}{q}\right\rfloor.$$
The two step functions change only at the four boundary points appearing in the definition of $I_q$. Comparing their values on the resulting intervals gives $1$ precisely for $r\in I_q$ and $0$ otherwise.
\end{proof}

We can express the full $p$-adic valuation in terms of $I_q$.

\begin{proposition}\label{prop:padic}
Let $p\geq5$ be prime. Then
$$\nuu_p(\TT_n)=\sum_{j\geq1}\ind_{[n+1]_{p^j}\in I_{p^j}}.$$
In particular, $p\mid\TT_n$  if and only if $\quad [n+1]_{p^j}\in I_{p^j}\text{ for at least one }j\geq1$.
\end{proposition}

\begin{proof}
Apply Lemma~\ref{lem:Interval} term by term to
\begin{align*}
\nuu_p(\TT_n)
&=\sum_{j\geq1}\left(
\left\lfloor\frac{3n+1}{p^j}\right\rfloor
-\left\lfloor\frac{2n+1}{p^j}\right\rfloor
-\left\lfloor\frac{n+1}{p^j}\right\rfloor\right)\\
&=\sum_{j\geq1}\left(
\left\lfloor\frac{3(n+1)-2}{p^j}\right\rfloor
-\left\lfloor\frac{2(n+1)-1}{p^j}\right\rfloor
-\left\lfloor\frac{n+1}{p^j}\right\rfloor\right)
\end{align*}
and the result is proved.
\end{proof}

We next pinpoint the boundary points that survive at every smaller $p$-adic scale. For $q=p^m$,
we write $I_q=[u_q+1,v_q-1]\cup[w_q+1,q-1]$,
with $u_q=\left\lfloor\frac{q+1}{3}\right\rfloor$, $v_q=\frac{q+1}{2}$, and $w_q=\left\lfloor\frac{2q+1}{3}\right\rfloor$.

\begin{lemma}\label{lem:boundary}
Let $p\geq5$ be prime and $q=p^m$. If $x\in\{u_q,v_q,w_q,q\}$, then $[x]_{p^j}\notin I_{p^j}$ for every $j\geq1$. Consequently, in the shifted variable $x=n+1$, the intervals
$[u_q+1,v_q-1]$ and $[w_q+1,q-1]$ correspond to maximal zero runs of $\TT_n\bmod p$.
\end{lemma}

\begin{proof}
First let $1\leq j\leq m$ and write $q=p^jP$. Since $P$ is odd, $P=2a+1$ and $v_q=a p^j+v_{p^j}$, 
so $v_q\equiv v_{p^j}\pmod{p^j}$, while $q\equiv0\pmod{p^j}$.
We consider two cases:
\begin{enumerate} 
\item if $P=3k+1$, then 
$u_q\equiv u_{p^j}\pmod{p^j}$, $w_q\equiv w_{p^j}\pmod{p^j}$.
\item if $P=3k+2$, then $u_q\equiv w_{p^j}\pmod{p^j}$, and $w_q\equiv u_{p^j}\pmod{p^j}$. 
\end{enumerate}
Thus every boundary residue belongs to
$\{0,u_{p^j},v_{p^j},w_{p^j}\}$, which is disjoint from $I_{p^j}=[u_{p^j}+1,v_{p^j}-1]\cup[w_{p^j}+1,p^j-1]$.

Also, If $j>m$, then each of $u_q,v_q,w_q,q$ is less than $p^j/3$ because $p\geq5$. Hence these residues lie below the first point of $I_{p^j}$.

Proposition~\ref{prop:padic} now shows that the four boundary points correspond to nonzero terms modulo $p$, whereas every point in either displayed open boundary interval lies in $I_q$ at scale $q$ and therefore corresponds to a zero term. Hence the intervals are indeed maximal zero runs and we are done.
\end{proof}

Define the run length
$$A(q):=v_q-u_q-1, \qquad B(q):=q-w_q-1.$$
These are the two new run lengths created at scale $q$ (except that $A(5)=0$ which corresponds to an empty interval).

The following proposition says that these lengths are increasing and the interval do not interferes with each other as the scales increase, and hence these lengths are the record values.

\begin{proposition}\label{prop:scale}
Let $q=p^m$ with $p\geq5$ prime. The positive zero-run lengths that first appear at scale $q$ are $B(q)$ and $A(q)$ (if $A(q)>0$). Moreover, $B(q)\geq A(q)$ and $$A(q)>B(p^{m-1})$$ for $m\ge 2$. 
Consequently, each positive $A(p^m)$ and each $B(p^m)$ is a left-to-right record value, and there are no other record values.
\end{proposition}

\begin{proof}
From Lemma~\ref{lem:boundary} we know that the new forbidden intervals at scale $q$ are not cut by lower-scale forbidden intervals, so their lengths are exactly $A(q)$ and $B(q)$. All other runs present at scale $q$ are copies of runs already present at lower scales. 
Indeed, if we consider the scale $p^{m-1}$, the conditions coming from
the scales $p,\ldots,p^{m-1}$ depend only on the residue modulo $p^{m-1}$, so
every complete block of length $p^{m-1}$ carries a translated copy of the
zero-run pattern at scale $p^{m-1}$. These copies cannot merge across adjacent
blocks, since every block boundary $x=ap^{m-1}$ satisfies
\[
[x]_{p^j}=0\notin I_{p^j},
\qquad 1\le j\le m-1,
\]
and hence corresponds to a nonzero term in the safe regions. In
addition, Lemma~5.3 shows that $u_q,v_q,w_q,q$ are nonzero at every
lower scale. Thus the two new forbidden intervals are separated from
all lower-scale runs, and no additional run length is created at scale
$q$.

Since $q\equiv1$ or $5\pmod6$,
$$B(q)-A(q)=
\begin{cases}
(q-7)/6,&q\equiv1\pmod6,\\
(q+1)/6,&q\equiv5\pmod6,
\end{cases}
$$
which is nonnegative; equality occurs only at $q=7$. For $m\geq2$, it is clear that
$B(p^{m-1})<\frac{p^{m-1}}3=\frac{q}{3p}$, and $A(q)\geq\frac{q-5}{6}$,
and therefore
$$A(q)-B(p^{m-1}) >\frac{q-5}{6}-\frac{q}{3p} =\frac{(p-2)q-5p}{6p}>0,$$
since $q\geq p^2$ and $p\geq5$. Also, the record statement follows by induction on the scale.
\end{proof}

Finally we are able to state our main theorem: the record values in closed form. Note that the two congruence classes of primes modulo $3$ behave differently.

\begin{theorem}\label{thm:Runp1}
Let $p\geq5$ be a prime. 
\begin{enumerate}
\item If $p\equiv1\pmod3$, then the left-to-right record values of $\{\RR_i^{(p)}\}$ are precisely the positive numbers of the forms $\frac{p^r-1}{6}$ and $\frac{p^r-4}{3}$ for $r\geq 1$.
\item If $p\equiv 2\pmod3$, then the left-to-right record values of $\{\RR_i^{(p)}\}$ are precisely the following:
\begin{enumerate} 
\item For odd $r\geq1$, the record values at scale $p^r$ are
$\frac{p^r-5}{6}$ and $\frac{p^r-2}{3}$ (except for $(p,r)=(5,1)$ of $\frac{p^r-5}{6}$, which is $0$).
\item For even $r\geq2$ , the record values at scale  $p^r$ are
$\frac{p^r-1}{6}$ and $\frac{p^r-4}{3}.$
\end{enumerate}
\end{enumerate}
\end{theorem}

\begin{proof}
First the case $p\equiv1\pmod3$. Such a prime satisfies $p\equiv1\pmod6$, and therefore $p^r\equiv1\pmod6$ for every $r\geq1$. Hence $A(p^r)=\frac{p^r-1}{6}$ and $B(p^r)=\frac{p^r-4}{3}$ by applying Proposition~\ref{prop:scale} directly.

Next the case $p\equiv2\pmod3$. Now $p\equiv5\pmod6$, so 
$$
p^r\equiv
\begin{cases}
5\pmod6,&r\text{ odd},\\
1\pmod6,&r\text{ even},
\end{cases}
$$
and again the formulas follow from the definitions of $A(p^r)$ and $B(p^r)$ and Proposition~\ref{prop:scale}.
\end{proof}

For example, set $p=7 (\equiv 1 \bmod 3)$, one can check that the distinct record values begin
$$1,8,15,57,113,400,799,\ldots = B(7), A(7^2), B(7^2), A(7^3),B(7^3),  A(7^4),B(7^4), \dots$$
Or let $p=5 (\equiv 2 \bmod 3)$, then the distinct record values begin
$$1,4,7,20,41,104,207,520,\ldots = B(5), A(5^2), B(5^2), A(5^3),B(5^3),  A(5^4),B(5^4), \dots$$
Note that $A(5)=0$ corresponding to the empty set, and $A(7)=B(7)=1$.

\medskip
We next determine the generating function for the zero-run lengths. For $j\geq1$, define
$$
A_j:=
\begin{cases}
(p^j-1)/6,&p^j\equiv1\pmod6,\\
(p^j-5)/6,&p^j\equiv5\pmod6,
\end{cases}
\qquad 
and
\qquad 
B_j:=
\begin{cases}
(p^j-4)/3,&p^j\equiv1\pmod6,\\
(p^j-2)/3,&p^j\equiv5\pmod6.
\end{cases}
$$
Let
$$
E_j(y):=\ind_{A_j>0}y^{A_j}+y^{B_j}.
$$

\begin{theorem}\label{thm:specp}
For every prime $p\geq5$ and every $m\geq1$,
$$Z_{p,m}(y)=E_m(y)+\frac{p-1}{2}\sum_{j=1}^{m-1}\left(\frac{p+1}{2}\right)^{m-j-1}E_j(y).$$
\end{theorem}

\begin{proof}

Let $q=p^m$ and $s=p^{m-1}$. In the shifted coordinate $x=n+1$, define $P_m(y)$ to be the spectrum contributed by the safe region $1\leq x\leq u_q$, and $Q_m(y)$ to be the spectrum contributed by $v_q\leq x\leq w_q$. Lemma~\ref{lem:boundary} says that these cuts do not split lower-scale zero runs. The two new forbidden intervals contribute exactly $E_m(y)$, and hence
\[
Z_{p,m}(y)=P_m(y)+Q_m(y)+E_m(y).
\]

We compute $P_m+Q_m$ by dividing the safe regions into blocks of length $s$.

\begin{enumerate}
\item 
If $p=6k+1$, then
$u_q=2ks+u_s, v_q=3ks+v_s$, and $w_q=4ks+w_s$. Thus
\[
P_m(y)=2kZ_{p,m-1}(y)+P_{m-1}(y).
\]
The middle safe region consists of a terminal partial block, $k-1$ complete blocks, and an initial partial block. Consequently,
\begin{align*}
Q_m(y) &=(k-1)Z_{p,m-1}(y)+P_{m-1}(y) +2Q_{m-1}(y)+E_{m-1}(y).
\end{align*}
Adding these identities and using
$P_{m-1}(y)+Q_{m-1}(y)=Z_{p,m-1}(y)-E_{m-1}(y)$
gives
$$P_m(y)+Q_m(y)=\frac{p+1}{2}Z_{p,m-1}(y)-E_{m-1}(y).$$
\item  If $p=6k-1$, then
$u_q=(2k-1)s+w_s$, $v_q=(3k-1)s+v_s$, and $w_q=(4k-1)s+u_s$. The same block decomposition gives
\begin{align*}
P_m(y)&=(2k-1)Z_{p,m-1}(y)+P_{m-1}(y)+Q_{m-1}(y)
+\ind_{A_{m-1}>0}y^{A_{m-1}},\\
Q_m(y)&=(k-1)Z_{p,m-1}(y)+P_{m-1}(y)+Q_{m-1}(y)+y^{B_{m-1}},
\end{align*}
so again
$P_m(y)+Q_m(y)=\frac{p+1}{2}Z_{p,m-1}(y)-E_{m-1}(y).$
\end{enumerate}
Anyway, in both cases we obtain the recurrence
$$Z_{p,m}(y) =\frac{p+1}{2}Z_{p,m-1}(y)+E_m(y)-E_{m-1}(y).$$
Since $Z_{p,1}(y)=E_1(y)$, iterating the recurrence yields
$$Z_{p,m}(y) =E_m(y)+\frac{p-1}{2} \sum_{j=1}^{m-1}\left(\frac{p+1}{2}\right)^{m-j-1}E_j(y),$$
and the proof is completed.
\end{proof}

For example, for $p=7$ and $m=3$, we have $E_1(y)=2y$, $E_2(y)=y^8+y^{15}$, and $E_3(y)=y^{57}+y^{113}$,
hence $$Z_{7,3}(y)=24y+3y^8+3y^{15}+y^{57}+y^{113}.$$

For $p=5$ and $m=3$, we have $E_1(y)=y$, $E_2(y)=y^4+y^7$, $E_3(y)=y^{20}+y^{41}$, and hence
$$Z_{5,3}(y)=6y+2y^4+2y^7+y^{20}+y^{41}.$$

\section{The number of nonzero entries}\label{sec:support}
Since from $Z_{p,m}$ we know the distribution of zero runs, we can also compute the number of terms not divisible by $p$ in the interval $0\leq n<p^m$. For a prime $p$ and $m\geq1$, let
$$N_{p,m}:=\#\{0\leq n<p^m:p\nmid\TT_n\}.$$

\begin{theorem}\label{thm:supportcount}
For every $m\geq1$, $N_{2,m}=\FF_{m+1}$, $N_{3,m}=\binom{m+2}{2}$, and for every prime $p\geq5$,
$$N_{p,m}=\frac{p+3}{2}\left(\frac{p+1}{2}\right)^{m-1}.$$

\end{theorem}

\begin{proof}
For $p=2$, Theorem~\ref{thm:mod2} identifies the nonzero entries with binary strings of length $m$ that end in $0$ and contain no consecutive $1$'s and hence $N_{2,m}=\FF_{m+1}$.

For $p=3$, Theorem~\ref{thm:mod3} identifies the nonzero entries with weakly increasing words of length $m$ over $\{0,1,2\}$, Hence
$N_{3,m}=\binom{m+2}{2}$.

Now let $p\geq5$ and write $z_{p,m}:=Z'_{p,m}(1)$ which counts the zeros in the interval and so $N_{p,m}=p^m-z_{p,m}$.
For both possible residues $p^m\equiv1,5\pmod6$,
\[
E_m'(1)=A_m+B_m=\frac{p^m-3}{2}.
\]
Differentiating the recurrence from the proof of Theorem~\ref{thm:specp} gives
\[
z_{p,m}=\frac{p+1}{2}z_{p,m-1}+\frac{p-1}{2}p^{m-1}.
\]
Subtracting this from $p^m$ yields
\[
N_{p,m}=\frac{p+1}{2}N_{p,m-1}.
\]
At $m=1$, we have
$N_{p,1}=p-E_1'(1)=\frac{p+3}{2}$
and the stated formula follows.
\end{proof}

In other words, we have Fibonacci growth for $p=2$, polynomial growth for $p=3$, and geometric growth with ratio $(p+1)/2$ for $p\geq5$. In all three cases, the density of nonzero entries tends to zero as $\lim_{m\to\infty}\frac{N_{p,m}}{p^m}=0$ in all three cases. 

\section{Concluding remarks}\label{sec:conclusion}

In this paper we fully determine the record values of the zero-run lengths of the $(3,2)$-Raney numbers modulo any given prime $p$. For every $m$, we also determine exactly the generating function for the zero-run lengths among the first $p^m$ terms, as well as the number of terms not divisible by $p$.

A central feature of the results is the distinction among the three cases $p=2$, $p=3$, and $p\ge5$.
For $p=2$, the support is described by even Fibbinary integers, and the record structure involves Fibonacci and Jacobsthal numbers. For $p=3$, the support is described by weakly increasing ternary words. For $p\ge5$, neither of these digit-word descriptions persists. Instead,
the support is controlled by a hierarchy of forbidden residue intervals modulo $p,p^2,p^3,\ldots$.
The same trichotomy is reflected in the number of nonzero terms: Theorem~6.1 gives Fibonacci growth for $p=2$, polynomial growth for $p=3$, and geometric growth with ratio $(p+1)/2$ for $p\ge5$. Nevertheless, in every case the density of nonzero terms tends to zero.

It would be interesting to understand whether the even-Fibbinary criterion, the weakly increasing ternary criterion, and the multiscale residue intervals have direct interpretations on ordered pairs of ternary trees or on the related marked fighting-fish models.

The present paper only considers the Raney sequence $R_{3,2}(n)$ whereas the $R_{2,1}(n)$ counterpart was treated in~\cite{AK_73}.  The natural next step is to study the zero-run spectra of the general Raney numbers. Based on computer experiment, we have a promising conjecture:

\begin{conjecture}\label{conj:raney}
Fix $k,r$ and a prime $p\nmid kr$. The left-to-right record values of the zero-run sequence of $R_{k,r}(n)\bmod p$ belong to a finite union of families of the form
\[
\frac{a p^m+b}{c},
\]
where $a,b,c$ depend only on $k,r$, and $p$. 
\end{conjecture}

Finally, it seems quite interesting and may be challenging to see whether in any of the combinatorial structures mentioned in the introduction there are combiatorial interpretations explaining the number theoretical results presented in this paper. 

\begin{problem}\label{prob:combexp}
Give combinatorial interpretations of the number-theoretic results obtained in this paper.
\end{problem}

We leave these directions to the interested readers.\\


\section*{Declaration of generative AI and AI-assisted technologies in the manuscript preparation process}
During the preparation of this work, the authors used ChatGPT (OpenAI) for language editing, organization of the exposition, and exploratory assistance with mathematical arguments and computations. All AI-assisted content was critically reviewed, revised, and independently verified by the authors. The authors take full responsibility for the content of the article.

\end{document}